\documentclass[12pt,a4paper]{article}

\usepackage[usenames]{color}
\usepackage{amssymb}
\usepackage{amsmath}
\usepackage{amsthm}
\usepackage{amsfonts}
\usepackage{amscd}
\usepackage{graphicx}

\usepackage{mathtools}

\DeclarePairedDelimiter\floor{\lfloor}{\rfloor}

\usepackage[colorlinks=true,
linkcolor=webgreen,
filecolor=webbrown,
citecolor=webgreen]{hyperref}

\definecolor{webgreen}{rgb}{0,.5,0}
\definecolor{webbrown}{rgb}{.6,0,0}

\usepackage{color}
\usepackage{fullpage}
\usepackage{float}

\usepackage{graphics}
\usepackage{latexsym}
\usepackage{epsf}

\newcommand{\seqnum}[1]{\href{https://oeis.org/#1}{\underline{#1}}}

\DeclarePairedDelimiter{\abs}   {\lvert } {\rvert }

\DeclareMathOperator\artanh{artanh}

\theoremstyle{plain}
\newtheorem{theorem}{Theorem}

\theoremstyle{definition}

\theoremstyle{remark}
\newtheorem{remark}[theorem]{Remark}

\title{The ordered Bell numbers as weighted sums of odd or even Stirling numbers of the second kind}
\author{Jacob Sprittulla \\ sprittulla@alice-dsl.de}
\date{\today}

\begin{document}
	
\maketitle
	
\begin{abstract}
	For the Stirling numbers of the second kind $S(n,k)$ and the ordered Bell numbers $B(n)$, we prove the identity $\sum_{k=1}^{n/2} S(n,2k)(2k-1)! = B(n-1)$. An analogous identity holds for the sum over odd $k$'s. 
\end{abstract}

\section{Introduction}
	For integers $0 \leq k \leq n$, we denote by $S(n,k)$ the Stirling numbers of the second kind and by $B(n)$ the ordered Bell numbers. These sequences can be found in the On-Line Encyclopedia of Integer Sequences (OEIS) \cite{OEIS} as \seqnum{A008277} and \seqnum{A000670}. They are related by
	\begin{align}
		\label{eq:sum1}
		B(n) = \sum_{k=0}^n k! S(n,k)   
		\text{.}
	\end{align}
	In this short note, we will give four different ways to express the ordered Bell numbers as a  weighted sum of odd or even Stirling numbers of the second kind, where the mentioned parity refers to the second argument.
	
	\begin{theorem}
	\label{tm:4ways}
		For $n \geq 1$, we have
		\begin{align}
			\label{eq:sumk0}
			B(n) 
			&= (-1)^{n+1} + 2 \sum_{k=0}^{\floor{n/2}} (2k)! S(n,2k)   
			= (-1)^n     + 2 \sum_{k=0}^{\floor{n/2}} (2k+1)! S(n,2k+1)   \\
			\label{eq:sumk1}
			&=\sum_{k=1}^{\floor{(n+1)/2}} (2k-1)! S(n+1,2k)
            =\sum_{k=0}^{\floor{(n+1)/2}} (2k)! S(n+1,2k+1) 
			\text{.}
		\end{align}
	\end{theorem}
	We could not find the Identities \eqref{eq:sumk1} (respectively \eqref{eq:sum4} below) in the literature. The Identities \eqref{eq:sumk0} are direct consequences of known other equations, as explained below.
		
\section{Proof of the Theorem}		
	The formal (exponential) generating functions (gf's) of $B(n)$ and $S(k,n)$ (for fixed $k$) are given by
	\begin{align}
		\label{eq:gfB}
		\sum_{n=0}^\infty \frac{x^n}{n!} B(n)
		&= \frac{1}{2-e^x} 
		=: \mathcal{B}(x)	\\	
		\label{eq:gfS}
		\sum_{n=0}^\infty \frac{x^n}{n!} S(n,k) 
		&= \frac{(e^x-1)^k}{k!}  
		\text{,}
	\end{align}
	see Quaintance and Gould \cite{QGo16}. The following identities are known 
	\begin{align}
		\label{eq:sum2}
		\sum_{k=0}^n (-1)^k k! S(n,k) &= (-1)^n  \\
		\label{eq:sum3}
		\sum_{k=1}^n (-1)^k (k-1)! S(n,k) &=
		\begin{cases}
			-1, & \text{if $n=1$;}\\
			 0, & \text{if $n\geq 2$, }
		\end{cases}   
		\text{,}
	\end{align}
	see Boyadzhiev \cite[Appendix A]{Boy18}. Combining identities \eqref{eq:sum1} and \eqref{eq:sum2} gives
	\begin{align*}
		\sum_{k=0}^{\floor{n/2}} (2k)! S(n,2k) 
		&= \tfrac{1}{2} \left(B(n) +(-1)^n\right) \\
		\sum_{k=0}^{\floor{n/2}} (2k+1)! S(n,2k+1) 
		&= \tfrac{1}{2} \left(B(n) -(-1)^n\right)
		\text{.}
	\end{align*}	
	This shows Equation \eqref{eq:sumk0} of Theoreom \ref{tm:4ways}.
	
	We show below the following theorem.
	
	\begin{theorem}
	\label{tm:Hn}
		Put $H(n):=\sum_{k=1}^n (k-1)! S(n,k)$ for $n \geq 1$. Then
		\begin{align}
			\label{eq:sum4}
			H(n) &=
			\begin{cases}
			1, & \text{if $n=1$;}\\
			2 B(n-1),  & \text{if $n\geq 2$. }
			\end{cases}   
		\end{align}
	\end{theorem}

	Combining identities \eqref{eq:sum3} and \eqref{eq:sum4} gives, for $n \geq 2$,	
	\begin{align}
		\label{eq:sum5}
		\sum_{k=1}^{\floor{n/2}} (2k-1)! S(n,2k)
		= \sum_{k=0}^{\floor{n/2}} (2k)! S(n,2k+1) 
		= B(n-1)
		\text{,}
	\end{align}	
	and therefore \eqref{eq:sumk1} of Theoreom \ref{tm:4ways}.
	It remains to prove \eqref{eq:sum4}.
	\begin{proof}[Proof of Theorem \ref{tm:Hn}]
		The gf $\mathcal{H}(x)$ of $H(n)$ is given by 
		\begin{align*}
		\mathcal{H}(x)	
		&=\sum_{n=0}^\infty \frac{x^n}{n!} \sum_{k=1}^n (k-1)! S(n,k)    
		 =\sum_{k=1}^\infty \frac{1}{k} (e^x-1)^k  \\
		&=-\log(2-e^x)
		\text{,}
		\end{align*}
		where we used \eqref{eq:gfS} and the Taylor series 
		$\log(1-z)=-\sum_{k=1}^{\infty} \frac{1}{k} z^{k}$ for $\abs{z}<1$. 
		Further, defining $G(1)=2$ and $G(n)=H(n)$ for $n \geq 2$, we get 
		\begin{align*}
		\mathcal{G}(x) 
		:= \sum_{n=0}^\infty \frac{x^n}{n!} G(n) 
		= x - \log(2-e^x)
		\text{,}
		\end{align*}
		noticing that $H(1)=1$.
		By elementary differentiation rules and \eqref{eq:gfB}, it follows that 
		$\mathcal{G}'(x)=\frac{2}{e^x-1}=2\mathcal{B}(x)$.
		Hence, by Wilf \cite[Formula 2.3.1]{Wil90}, we get $H(n)=G(n) = 2 B(n-1)$ for $n \geq 2$. 
	\end{proof} 

	\begin{remark}
		We put 
		$H_e(n):=\sum_{k=1}^{\floor{n/2}} (2k-1)! S(n,2k)$ 
		and 
		$H_o(n):=\sum_{k=0}^{\floor{n/2}} (2k)! S(n,2k+1)$.
		Using the same tools as in the proof above, we can also show that the corresponding gf's are given by
		\begin{align*}
			\sum_{n=0}^\infty \frac{x^n}{n!} H_e(n) 
			&=\artanh(e^x-1)  
			= \frac{1}{2} \log \left( \frac{e^x}{2-e^x} \right)  \\
			\sum_{n=0}^\infty \frac{x^n}{n!} H_o(n)
			&=-\frac{1}{2} \log \left( e^x (2-e^x) \right) 
			\text{.}
		\end{align*}
		With these gf's, we can deduce \eqref{eq:sum5} directly, without using \eqref{eq:sum3}.
	\end{remark}
	
	\begin{remark}
		The quantities $W(n,k):=k! S(n+1,k+1)$ are called the \textit{Worpitzky numbers}, see Vanderfelde \cite{Wil90} (OEIS \seqnum{A130850}). With this notation, \eqref{eq:sum5} can be written as ($n \geq 1$)
		\begin{align*}
			\sum_{k=1}^{\floor{n/2}} W(n,2k)
			= \sum_{k=0}^{\floor{n/2}} W(n,2k+1)
			= B(n)
			\text{.}
		\end{align*}
	\end{remark}	
	
\bibliographystyle{plain}

\bigskip
\hrule
\bigskip

\noindent 2010 {\it Mathematics Subject Classification}: Primary 11A51;
Secondary 05A17.

\noindent \emph{Keywords:}
ordered Bell number, Stirling number of the second kind, Worpitzky number triangle.

\bigskip
\hrule
\bigskip
		
\end{document}